\documentclass[orivec,runningheads,envcountsame,verbatim]{llncs}

\usepackage{graphicx}
\usepackage{amsmath}
\usepackage{amssymb}

\newcommand{\quantq}{\mathbf{Q}}

 \newcommand{\skilla} {\mbox{\bf \#}}
 \newcommand{\skillb} {\mbox{\bf !}}
 \newcommand{\bang}[1]{\mbox{\bf !}^{#1}}

\newcommand{\xpcp}{\mbox{\small PCP}}

\newcommand{\sigf}[3]{ \Sigma_{#1,#2,#3} }
\newcommand{\msigf}[4]{\Sigma_{#1,#2,#3}^\mathfrak{#4}}

\begin{document}             % End of preamble and beginning of text.

\newcommand{\ccom}[1]{\mbox{\footnotesize{#1}}}

\setlength{\parindent}{0mm}
 \setlength{\parskip}{4pt}

\title{Notes on Fragments of First-Order Concatenation Theory}

\author{Lars Kristiansen\inst{1,2} \and Juvenal Murwanashyaka\inst{1}}

 \institute{
Department of Mathematics, University of Oslo, Norway
\and
    Department of Informatics, University of Oslo, Norway
}

 \maketitle                   % Produces the title.

\newcommand{\xd}{\texttt{D}}

\newcommand{\integer}{\ensuremath{\mathbb Z}}
\newcommand{\rational}{\ensuremath{\mathbb Q}}
\newcommand{\nat}{\ensuremath{\mathbb N}}
\newcommand{\real}{\ensuremath{\mathbb R}}
\newcommand{\kleeneT}{{\mathcal T}}

\newcommand{\kleeneU}{{\mathcal U}}

\begin{abstract}
We identify a number of decidable and undecidable fragments of first-order concatenation theory.
We also give  a purely universal axiomatization which is complete for the fragments we identify.
Furthermore, we prove some normal-form results.
\end{abstract}

\section{Introduction}

\subsection{The Purpose of These Notes}

The purpose of this paper is to give full proofs of results published elsewhere.

\subsection{First-order Concatenation theory} 
First-order concatenation theory can be compared to first-order number theory, e.g.,  Peano Arithmetic or Robinson Arithmetic.
The universe of a standard structure for first-order number theory is the set of natural numbers.
The universe of a standard structure for first-order concatenation theory is a set of strings over some alphabet.
A first-order language for number theory normally contains two binary functions symbols. In a standard structure these 
symbols will be interpreted as addition and multiplication. A first-order language for
  concatenation theory normally contains just one binary function symbol.  In a standard structure this  symbol will be interpreted
as the operator that concatenates  two stings. A classical first-order language for concatenation theory
contains no other non-logical symbols apart from constant symbols.

In this paper we extend concatenation theory with a binary relation symbol and introduce bounded  quantifiers 
analogous to  the bounded quantifiers $(\forall x\leq t)\phi$ and $(\exists x\leq t)\phi$ we know
from number theory. Before we go on and  state our main results, we
will explain some notation and state a few basic definitions.

\subsection{Notation and Basic Definitions}

We will use $\boldsymbol{0}$ and $\boldsymbol{1}$ to denote respectively the bits zero and  one,
and we use pretty standard notation when we work with bit strings: $\{ \boldsymbol{0},  \boldsymbol{1}\}^*$ denotes the set of all finite bit strings; $|b|$ denotes the 
length of the bit string $b$; $(b)_i$ denotes the $i^{\mbox{{\scriptsize th}}}$ bit of the bit string $b$; and 
$\boldsymbol{0}\boldsymbol{1}^3\boldsymbol{0}^2\boldsymbol{1}$ denotes the bit string $\boldsymbol{0}\boldsymbol{1}\boldsymbol{1}\boldsymbol{1}\boldsymbol{0}\boldsymbol{0}\boldsymbol{1}$.
The set $\{ \boldsymbol{0},  \boldsymbol{1}\}^*$
contains the empty string which we will denote $\varepsilon$.

Let $\mathcal{L}_{BT}$ denote the first-order language that consist
of the constants symbols $e,0,1$, the binary function symbol $\circ$
and the binary relation symbol $\sqsubseteq$.
We will consider two $\mathcal{L}_{BT}$-structures named  $\mathfrak{B}$ and $\mathfrak{D}$.

The universe of $\mathfrak{B}$ is the set $\{ \boldsymbol{0},  \boldsymbol{1}\}^*$. 
The constant symbol $0$ is interpreted as the string containing nothing but the bit $\boldsymbol{0}$, 
and the constant symbol $1$ is interpreted as the string
containing nothing but the bit $\boldsymbol{1}$, that is, $0^{\mathfrak{B}} = \boldsymbol{0}$ and  
$1^{\mathfrak{B}} = \boldsymbol{1}$. The constant symbol $e$ is interpreted as the empty string, that is, $e^{\mathfrak{B}} = \varepsilon$. 
Moreover, $\circ^\mathfrak{B}$ is
the function that concatenates two strings 
(e.g. $\boldsymbol{0}\boldsymbol{1} \circ^\mathfrak{B}  \boldsymbol{0}\boldsymbol{0}\boldsymbol{0} 
=  \boldsymbol{0}\boldsymbol{1}\boldsymbol{0}\boldsymbol{0}\boldsymbol{0}$  and
$ \varepsilon \circ^\mathfrak{B}  \varepsilon    = \varepsilon$). Finally, $\sqsubseteq^\mathfrak{B}$ is the substring relation,
that is, $u \sqsubseteq^\mathfrak{B} v$ iff there exists bit strings $x,y$ such that $xuy=v$. 

The structure $\mathfrak{D}$ is the same structure as $\mathfrak{B}$ with one exception: the relation
$u \sqsubseteq^\mathfrak{D} v$ holds iff $u$ is a prefix of $v$, that is, iff 
 there exists a bit string $x$ such that $ux=v$. To improve the readability we will use the  symbol $\preceq$
in place of the symbol $\sqsubseteq$ when we are working in the structure  $\mathfrak{D}$. Thus, $u \sqsubseteq v$
should be read as ``$u$ is a substring of $v$'', whereas  $u \preceq v$
should be read as ``$u$ is a prefix of $v$''. When we do not have a particular structure in mind, e.g.
when we deal with syntactical matters, we will stick to the symbol $\sqsubseteq$.

We  introduce the {\em bounded quantifiers} $(\exists x  \sqsubseteq t) \alpha$ and $(\forall x  \sqsubseteq t) \alpha$
as syntactical abbreviations for receptively $(\exists x)[x  \sqsubseteq t \, \wedge \, \alpha]$
and $(\forall x)[x  \sqsubseteq t \, \rightarrow \, \alpha]$ ($x$ is of course not allowed to occur in the term $t$), 
and we define the $\Sigma$-formulas inductively
by
\begin{itemize}
\item $\alpha$ and $\neg \alpha$ are $\Sigma$-formulas if $\alpha$ is  of the form $s\sqsubseteq t$ or of the form 
$s = t$ where $s$ and $t$ are terms
\item $\alpha \vee \beta$ and $\alpha \wedge \beta$ are  $\Sigma$-formulas if $\alpha$ and $\beta$ are $\Sigma$-formulas
\item  $(\exists x  \sqsubseteq t) \alpha$ and $(\forall x  \sqsubseteq t) \alpha$ and $(\exists x) \alpha$ are 
 $\Sigma$-formulas if $\alpha$ is a $\Sigma$-formula.
\end{itemize}

We assume that the reader notes the similarities with first-order number theory.
The formulas that correspond to $\Sigma$-formulas in number theory are often called $\Sigma_1$-formulas or $\Sigma^0_1$-formulas.
Next we introduce the  biterals. The biterals correspond to the numerals of first-order number theory.
Let $b$ be a bit string. We define the {\em biteral} $\overline{b}$ by 
$\overline{\varepsilon} = e$, $\overline{b\boldsymbol{0}} = \overline{b} \circ \boldsymbol{0}$ and $\overline{b\boldsymbol{1}} = \overline{b} \circ \boldsymbol{1}$.

A $\Sigma$-formula
$\phi$ is called a $\sigf{n}{m}{k}$-formula if it contains $n$ unbounded existential quantifiers, $m$ bounded existential quantifiers and $k$ bounded universal quantifiers. A {\em sentence} is a formula with no free variables. 
The {\em fragment} $\msigf{n}{m}{k}{B}$ ($\msigf{n}{m}{k}{D}$) is the set of $\sigf{n}{m}{k}$-sentences 
true in $\mathfrak{B}$ (respectively, $\mathfrak{D}$).

To improve the readability we may 
 skip the operator $\circ$ in first-order formulas and simply write $st$ in place of $s\circ t$.
Furthermore,
we will  occasionally contract quantifiers and 
write, e.g., $\forall w_{1}, w_{2} \sqsubseteq u[\phi]$ in place of $(\forall w_{1} \sqsubseteq u)(\forall  w_{2} \sqsubseteq u)\phi$, and 
for $\sim \, \in \! \{\preceq, \sqsubseteq, =\}$, we  will sometimes write $s\not \sim t$ in place of $\neg s\sim t$.

\subsection{Main Results and Related Work}

We prove that the fragment  $\msigf{0}{m}{k}{B}$ is decidable (for any $m,k\in\nat$), and we prove that
$\msigf{1}{2}{1}{B}$ and $\msigf{1}{0}{2}{B}$ are undecidable. Furthermore, we prove that the fragments
 $\msigf{0}{m}{k}{D}$ and  $\msigf{n}{m}{0}{D}$ are decidable (for any $n,m,k\in\nat$), and we prove that
 $\msigf{3}{0}{2}{D}$  and  $\msigf{4}{1}{1}{D}$  are undecidable.
Our results on decidable fragments are corollaries of theorems that have an interest in their own right:
We prove the existence of   normal forms, and we give a purely universal axiomatization of concatenation theory which is
$\Sigma$-complete.

Recent related work can be found in  Halfon et al.\ \cite{halfon}, 
Day et al.\ \cite{day},  Ganesh et al.\ \cite{ganesh}, Karhum\"aki et al.\ \cite{karh}  and
several other places, see Section 6 of \cite{ganesh} for further references.

The material in Section 8 of the textbook Leary \& Kristiansen \cite{leary} is also
 related to the research presented in this paper.
So is a series of papers that starts with with Grzegorczyk \cite{grz} 
and includes Grzegorczyk \& Zdanowski \cite{zd}, Visser \cite{visser} and
Horihata \cite{hori}.
These papers deal with the essential undecidability\footnote{A first-order theory is {\em essentially undecidable}
 if the theory---and every extension of the theory---is undecidable. Tarski \cite{utarski} is a very readable introduction
to the subject.} of various first-order theories of concatenation.
The relationship between the various axiomatizations of concatenation theory we find in these  papers and the axiomatization we
give below has not yet been investigated.

The theory of concatenation seems to go back to 
work of Tarski \cite{tarski} and Quine \cite{quine}, see Visser \cite{visser} for a brief account of its history.

\section{$\Sigma$-complete Axiomatizations}

\begin{definition}
The first-order theory $B$ contains the following eleven non-logical axioms:
\begin{enumerate} 
\item $\forall x[ \ x= e  x \wedge x=x  e \ ] $
\item $\forall x y z [ \ (x  y)  z = x  (y   z) \ ]$
\item $\forall x y[ \ ( x \neq y) \to 
( \ ( x   0 \neq y   0)  \wedge ( x   1 \neq y   1) \ ) \ ] $
\item $ \forall x y [ \ x   0 \neq y   1 \ ] $
\item $\forall x [ \ x \sqsubseteq e \leftrightarrow x=e \ ] $
\item $\forall x[ \ x \sqsubseteq 0 \leftrightarrow (x=e \vee x = 0) \ ] $
\item $\forall x [ \ x \sqsubseteq 1 \leftrightarrow (x=e \vee x = 1) \ ] $
\item $ \forall x y [ \ x \sqsubseteq 0  y   0 \leftrightarrow (x = 0  y   0 \vee x \sqsubseteq 0  y  \vee x \sqsubseteq y  0 ) \ ]     $
\item $\forall x y [ \ x \sqsubseteq 0  y   1 \leftrightarrow (x = 0  y   1 \vee x \sqsubseteq 0  y  \vee x \sqsubseteq y  1 ) \ ]     $
\item $ \forall x  y [ \ x \sqsubseteq 1  y   0 \leftrightarrow (x = 1  y   0 \vee x \sqsubseteq 1  y  \vee x \sqsubseteq y  0 ) \ ]     $
\item $ \forall x y [ \ x \sqsubseteq 1  y   1 \leftrightarrow (x = 1  y   1 \vee x \sqsubseteq 1  y  \vee x \sqsubseteq y  1 ) \ ]     $
\end{enumerate} 
We will use $B_i$ to refer to the $i^{\mbox{{\scriptsize th}}}$ axiom of $B$.
\end{definition}

\begin{theorem}[$\Sigma$-completeness of $B$]\label{bsigcomp}
For any $\Sigma$-sentence $\phi$, we have
$$
 \mathfrak{B} \models \phi  \; \Rightarrow \; B\vdash \phi \; .
$$
\end{theorem}

\begin{proof} (Sketch)
Prove (by induction on the structure of $t$) that there for any variable-free $\mathcal{L}_{BT}$-term $t$ exists a biteral $b$ such that 
\begin{align}
 \mathfrak{B} \models t=b  \; \Rightarrow \; B\vdash t=b\; . \label{eeea}
\end{align}
Prove (by induction on the structure of $b_2$) that we for any biterals $b_1$ and $b_2$ have
\begin{align}
 \mathfrak{B} \models  b_1\neq b_2  \; \Rightarrow \; B\vdash  b_1 \neq b_2 \; . \label{eeeb}
\end{align}
Use  $B\vdash \forall x[ x0\neq e \, \wedge \, x1\neq e]$ when proving (2).
Furthermore, prove 
 (by induction on the structure of $b_2$) that we for any biterals $b_1$ and $b_2$ have
\begin{align}
 \mathfrak{B} \models  b_1\sqsubseteq  b_2  \; \Rightarrow \; B\vdash  b_1 \sqsubseteq b_2  
\;\;\;\; \mbox{ and } \;\;\;\;  \mathfrak{B} \models  b_1\not\sqsubseteq b_2  \; \Rightarrow \; B\vdash  b_1 \not\sqsubseteq b_2 \; . \label{eeec}
\end{align}
It follows from (\ref{eeea}), (\ref{eeeb}) and (\ref{eeec}) that we have
\begin{align}
 \mathfrak{B} \models  \phi  \; \Rightarrow \; B\vdash  \phi \; . \label{eeed}
\end{align}
for any $\phi$  of one of the four forms $t_1=t_2$, $t_1\neq t_2$, $t_1 \sqsubseteq t_2$, and $t_1 \not\sqsubseteq t_2$ where 
$t_1$ and $t_2$ are variable-free terms.

Use induction on the structure of $b$ to prove the following claim:
\begin{quote} 
\textit{If $\phi(x)$ is an $\mathcal{L}_{BT}$-formula such that we have
$\mathfrak{B} \models \phi(b)  \Rightarrow  B \vdash \phi(b)$ 
for any biteral $b$, 
then we also have
$$\mathfrak{B} \models (\forall x \sqsubseteq b)\phi(x)  \Rightarrow 
B \vdash (\forall x \sqsubseteq b)\phi(x)$$
for any biteral $b$.}
\end{quote}

Finally, prove (by induction on the structure of $\phi$)   that we for any $\Sigma$-sentence $\phi$ have
 $\mathfrak{B} \models \phi  \Rightarrow  B \vdash \phi$. 
Use (\ref{eeed}) in the base cases, that is, when $\phi$ is an atomic sentence or a negated atomic sentence.
Use the claim and (1) in the case $\phi$ is of the form $(\forall x \sqsubseteq t)\psi$. The remaining cases are rather straightforward.
\qed
\end{proof}

A detailed proof of Theorem \ref{bsigcomp} can be found in Section \ref{bsigcompproof}.

\begin{definition}
The first-order theory $D$ contains the following seven non-logical axioms:
\begin{enumerate} 
\item[-] the first four axioms are the same as the first four axioms of $B$
\item[5.] $\forall x[ \ x \preceq e \leftrightarrow x=e \ ]$
\item[6.] $ \forall x y [ \  x \preceq y \circ 0 \leftrightarrow 
( x = y \circ 0 \vee x \preceq y) \ ]$ 
\item[7.] $\forall x y [ \  x \preceq y \circ 1  \leftrightarrow 
(x = y \circ 1 \vee x \preceq y) \ ]$
\end{enumerate}
We will use $D_i$ to refer to the $i^{\mbox{{\scriptsize th}}}$ axiom of $D$.
\end{definition}

The proof of the next theorem can be found in Section \ref{dsigcompproof}.
More material related to the theories $B$ and $D$
can   be found in Chapter 8 of Leary \& Kristiansen \cite{leary}.

\begin{theorem}[$\Sigma$-completeness of $D$]\label{dsigcomp}
For any $\Sigma$-sentence $\phi$, we have
$$
  \mathfrak{D} \models \phi \; \Rightarrow \;  D\vdash \phi\; .
$$
\end{theorem}

\begin{corollary}
The fragments $\msigf{0}{m}{k}{B}$ and $\msigf{0}{m}{k}{D}$ are decidable (for any $m,k\in\nat$).
\end{corollary}

\begin{proof}
We prove that $\msigf{0}{m}{k}{B}$ is decidable. Let $\phi$ be a $\sigf{0}{m}{k}$-formula.  The negation of a $\sigf{0}{m}{k}$-formula
is logically equivalent to a $\sigf{0}{k}{m}$-formula (by De Morgan's laws).
We can compute a  $\sigf{0}{k}{m}$-formula $\phi'$ which is logically equivalent to $\neg \phi$.
 By Theorem \ref{bsigcomp},
we have  $B\vdash \phi$ if $\mathfrak{B}\models \phi$, and we have $B\vdash \phi'$ if $\mathfrak{B}\models \neg \phi$. 
The set of formulas derivable from
the axioms of $B$ is computably enumerable. Hence it is decidable if  $\phi$ is true in $\mathfrak{B}$. The proof that the fragment $\msigf{0}{m}{k}{D}$
is decidable is similar.
\qed
\end{proof}

\section{Normal Forms}

Some of the lemmas below  are based on results and proofs  found in Senger \cite{senger} and B\"uchi \& Senger \cite{bs}.
They  prove that any $\Sigma$-formula in the language $\{\circ, 0,1,e\}$ is equivalent in $\mathfrak{B}|_{\{\circ, 0,1,e\}}$
to a formula of the form
$(\exists v_0)\ldots (\exists v_k)(s=t)$.

\begin{lemma}\label{sulten}
Let $\mathfrak{A} \in \lbrace \mathfrak{B}, \mathfrak{D} \rbrace$, and let $s_{1}, s_{2}, t_{1}, t_{2}$ be $\mathcal{L}_{BT}$-terms. 
We have 
$$\mathfrak{A} \models  (s_{1} = t_{1} \wedge s_{2} = t_{2}) \; \leftrightarrow \;
 s_{1}  0  s_{2}  s_{1}  1  s_{2} = t_{1}  0  t_{2}  t_{1}  1  t_{2}\; .$$ 
\end{lemma}

\begin{proof}
Assume $s_{1}  \boldsymbol{0}  s_{2}  s_{1}  \boldsymbol{1}  s_{2} = t_{1}  \boldsymbol{0}  t_{2}  t_{1}  \boldsymbol{1}  t_{2}$.
Then $|s_{1}  \boldsymbol{0}  s_{2}| = |t_{1}  \boldsymbol{0}  t_{2}|$ and $|s_{1}  \boldsymbol{1} \ s_{2}| = |t_{1}  \boldsymbol{1}  s_{2}|$.
The proof splits into the two cases $|s_1|=|t_1|$ and $|s_1|\neq |t_1|$. In the case when $|s_1|=|t_1|$, we obviously have $s_1=t_1$ and $s_2=t_2$.
Assume  $|s_1|\neq |t_1|$. We can w.l.o.g. assume that $|s_1|< |t_1|$. This implies that 
$$  \boldsymbol{0} \; = \; (s_{1}  \boldsymbol{0}  s_{2})_{|s_1| + 1}  \; = \; (t)_{|s_1| + 1} \; = \; (s_{1}  \boldsymbol{1}  s_{2})_{|s_1| + 1} \; = \;  \boldsymbol{1}\; .$$
This is a contradiction. This proves  the implication from the right to the 
left. The converse implication is obvious.
\qed
\end{proof}

\begin{lemma}\label{firecola}
Let $s_{1},s_{1},t_{1},t_{2}$  be $\mathcal{L}_{BT }$-terms.         
There exist $\mathcal{L}_{BT }$-terms $s ,t$ and variables $v_1,\ldots, v_k$ such that 
$$\mathfrak{D} \models (s_{1}\preceq t_{1} \vee s_{2} \preceq t_{2}) \leftrightarrow 
\exists v_{1} \ldots \exists v_{k}[s = t ]\; .$$
\end{lemma}

\begin{proof}
 Let $x_1,\ldots , x_6$ be variables that do not occur in any of the terms $s_1,s_2,t_1,t_2$.
It is not very hard to see that the formula 
$s_{1} \preceq t_{1} \vee  s_{2} \preceq t_{2}$
is equivalent in $\mathfrak{D}$ to the formula 
\begin{align*}
\exists x_{1} \ldots  x_{6} &
[ \ s_{1} = x_{1}  x_{2} \; \wedge \; t_{1} = x_{1}  x_{3} \; \wedge \; \\ & 
 \ s_{2} = x_{4}  x_{5} \; \wedge \;  t_{2} = x_{4}  x_{6} \; \wedge \; 
( x_{2} = e \; \vee \; x_{5} = e ) \ ] \; . \tag{*}
\end{align*}

Let $\psi(u,w)$ be the formula
\begin{align*}
  \exists y_1  y_2  y_3  y_4 [\
 y_1y_2 = 0 \; \wedge \; y_3y_4 = 1 
 \; \wedge \; 
 u y_1  w y_2 & = w   y_2  u y_1  \\ &  \; \wedge \;   
u y_3  w y_4  = w y_4  u y_3 \ ] \; .
\end{align*}

We claim that 
\begin{align*}
\mathfrak{D} \models   (  u = e \; \vee\; w = e  )  \; \leftrightarrow \; 
(  uw = wu   \; \wedge\; \psi(u,w)  )\; . \tag{**}
\end{align*}

We prove (**). Assume that $u = e  \vee w = e$. Let us say that $u = e$ (the case when $w = e$ is symmetric).
It is obvious that we have $uw = wu$. Moreover,  $\psi(u,w)$ holds with $y_1=y_3=e$, $y_2=0$ and $y_4=1$.
This prove the left-right implication of (**). 

To see that the converse implication holds, assume that
$\neg(u = e  \vee w = e)$, that is, both $u$ and $w$ are different from the empty string. Furthermore,
assume that $uw = wu$. We will argue that $\psi(u,w)$ does not hold: Since $uw = wu$ and both $u$ and $w$
contain at least one bit, it is  either the case that 0 is the last bit of both strings, or it is that case
that 1 is the last bit of both strings. If $0$ is the last bit of both, the two equations
$u y_3  w y_4  = w y_4  u y_3$ and  $y_3y_4 = 1$ cannot be satisfied simultaneously.
If $1$ is the last bit of both, the two equations  $u y_1  w y_2  = w   y_2  u y_1$
 and   $y_1y_2 = 0$ cannot be satisfied simultaneously. Hence we conclude that $\psi(u,w)$ does not hold.
This completes the proof of (**).

Our lemma follows from (*) and (**) by Lemma \ref{sulten}.
\qed
\end{proof}

\begin{lemma}\label{amanda}
Let $ \mathfrak{A} \in \lbrace \mathfrak{B}, \mathfrak{D} \rbrace$. 
Let $s_{1}, s_{2}, t_{1}, t_{2}$ be $\mathcal{L}_{BT}$-terms. 
There exist $\mathcal{L}_{BT}$-terms $s,t$ and variables $v_{0},\ldots ,v_{k}$ such that 
\begin{enumerate}
\item[(1)] $ \mathfrak{A} \models (s_{1} = t_{1} \vee s_{2} = t_{2}) \leftrightarrow 
\exists v_{0}\ldots  v_{k}[ s=t] $
\item[(2)] $ \mathfrak{A} \models  s_{1} \neq t_{1}  \leftrightarrow 
\exists v_{0} \ldots \exists v_{k} [ s=t ] $.
\end{enumerate}
\end{lemma}

\begin{proof}
Observe that  $s_{1} = t_{1} \vee s_{2} = t_{2}$ is equivalent in $\mathfrak{D}$ to
$$  (s_{1} \preceq t_{1} \wedge t_{1} \preceq s_{1} ) \vee 
(s_{2} \preceq t_{2} \wedge t_{2} \preceq s_{2} )$$
which again is  (logically) equivalent to
\begin{align*}
 (s_{1} \preceq t_{1} \; \vee \;  s_{2} \preceq t_{2} ) \; \wedge  \;
  (s_{1} \preceq t_{1} \;  \vee \;  & t_{2} \preceq s_{2} ) \; \wedge \; \\ & 
 \  ( t_{1} \preceq s_{1} \; \vee \; s_{2} \preceq t_{2} ) \; \wedge \;
( t_{1} \preceq s_{1} \; \vee \;  t_{2} \preceq s_{2} )\; .
\end{align*}
By Lemma \ref{sulten} and Lemma \ref{firecola}, it follows that (1) holds for
 the structure $\mathfrak{D}$. To see that  (1) also holds for the structure $\mathfrak{B}$,
observe that the relation $x \preceq^\mathfrak{D} y$ can be expressed  in $\mathfrak{B}$ by the formula  $\exists v[x v = y]$.

In order to see that (2) holds, observe that the formula  $s\neq t$ is equivalent---in both $\mathfrak{B}$ 
and $\mathfrak{D}$---to the formula
\begin{align*}\exists x y z [ \ s= t   0    x  \; \vee \; s= t   1    x &  \; \vee \;
 t= s   0    x \vee t= s   1    x  \; \vee \; \\ &
( s= x    1   y \; \wedge \;   t= x    0   z ) \; \vee \;   ( s= x    0   y \; \wedge \;  t= x    1   z )  \ ]\; .
\end{align*}
Thus,  (2) follows from (1) and Lemma \ref{sulten}.
\qed
\end{proof}

\begin{lemma}\label{martin}
Let $s_1, t_1$ be $\mathcal{L}_{BT}$-terms. There exist $\mathcal{L}_{BT}$-terms $s,t$ and variables $v_{1},\ldots, v_{k}$ such that 
\begin{enumerate}
\item[(1)] $ \mathfrak{D} \models s_{1} \preceq t_{1}  \leftrightarrow 
\exists v_{1}[ s_1 v_{1} = t_1 ] $
\item[(2)] $ \mathfrak{D} \models  s_{1} \not\preceq t_{1}  \leftrightarrow 
\exists v_{1} \ldots  v_{k} [s=t] $. 
\end{enumerate}
\end{lemma}

\begin{proof}
It is obvious that  (1) holds. Furthermore,
the formula $ s_{1} \not\preceq t_{1}$ is equivalent in $\mathfrak{D}$ to the formula 
\begin{align*}
( t_1 \preceq s_1 \; \wedge \;  t_1 \neq s_1  ) \; \vee \; 
\exists x y z [  
 (t_1 = x    0    y &  \; \wedge \; s_1 =x    1    z  ) \; \vee \; \\ & 
(t_1 = x    1    y \; \wedge \; s_1 =x    0    z  )]\;  .
\end{align*}
Thus,  (2) follows by  Lemma \ref{sulten}, Lemma \ref{amanda} and (1).
\qed
\end{proof}

\paragraph{Comment:} It is not known to us whether 
the bounded universal  quantifier that appears in clause (2)  of 
the next lemma can be eliminated. 

\begin{lemma}\label{tyggegummi}
Let $s_1,t_1$ be $\mathcal{L}_{BT}$-terms. There exist $\mathcal{L}_{BT}$-terms $s,t$ and variables $v_{1},\ldots, v_k$ such that 
\begin{enumerate}
\item[(1)] $ \mathfrak{B} \models s_1 \sqsubseteq t_1 \leftrightarrow 
\exists v_{1} \exists v_{2}[  t_1 = v_{1} s_1  v_{2}  ] $
\item[(2)] $\mathfrak{B} \models  s_1 \not\sqsubseteq t_1 \leftrightarrow 
\forall v_{1} \sqsubseteq t_1 \exists v_{2}\ldots v_{k}
[s = t ] $. 
\end{enumerate}
\end{lemma}

\begin{proof}
Cause (1)  is trivial. Furthermore, observe that
$s_1 \not\sqsubseteq t_1$ is equivalent in $\mathfrak{B}$ to 
$(\forall v  \sqsubseteq t_1)\alpha$ where $\alpha$ is
\begin{multline*}
\exists x [ \ t_1x = v s_1 \, \wedge \,  x\neq  e \ ]\;\vee \; \exists x y z [ \ 
(t_1 = x0y \, \wedge \, vs_1 = x1z) \; \vee \; \\ (t_1 = x1y \, \wedge \, vs_1 = x0z)   \  ]  \; .
\end{multline*}
If we let $vs_1 \preceq t_1$ abbreviate $(\exists x)(v s_1 x = t)$, then $\alpha$ can be written as
$vs_1 \not\preceq t_1$.
Thus, (2) follows by Lemma \ref{amanda}(2).
\qed
\end{proof}

\begin{theorem}[Normal Form Theorem I] \label{bispegaard}
Any $\Sigma$-formula $\phi$ is equivalent in $\mathfrak{D}$ to a $\mathcal{L}_{BT}$-formula $\phi'$ of the form 
$$ \phi' \; \equiv \; (\quantq_{1}^{t_{1}} v_{1} )\ldots (\quantq_{m}^{t_{m}} v_{m} )  (s=t)$$
where $t_{1},..,t_{m}, s,t$ are $\mathcal{L}_{BT}$-terms and 
$\quantq_{j}^{t_{j}} v_{j} \in \lbrace \exists v_{j}, \exists v_{j} \preceq t_{j} ,  \forall v_{j} \preceq t_{j} \rbrace$ 
for  $j=1,\ldots , m$.  
Moreover, if $\phi$ does not contain bounded universal quantifiers, then  $\phi'$  does not contain bounded quantifiers.
\end{theorem}

\begin{proof}
We proceed by induction on the structure of $\phi$ (throughout the proof we  reason in the structure $\mathfrak{D}$).
Suppose $\phi$ is an atomic formula or the negation of an atomic formula. 
If $\phi$ is of the form $s=t$, let $\phi^{\prime}$ be $s=t$. Use Lemma \ref{amanda}(2)
if $\phi$ is of the form $\neg s=t$.  Use Lemma \ref{martin} 
if $\phi$ is  of one of the forms $s\preceq t$ and $\neg s\preceq t$.

Suppose $\phi$ is of the form $\alpha \wedge \beta$. By our induction hypothesis, we have formulas
$$  \alpha^{\prime} \; \equiv \;  (\quantq_{1}^{t_{1}} x_{1})\ldots (\quantq_{k}^{t_{k}} x_{k})  (s_{1}=t_{1}) \;\; \mbox{ and } \;\;
 \beta^{\prime} \; \equiv \;  (\quantq_{1}^{s_{1}} y_{1})\ldots (\quantq_{m}^{s_{m}} y_{m}) (s_{2}=t_{2}) $$
which are equivalent  to respectively $\alpha$ and $\beta$. 
Thus, $\phi$ is equivalent to a formula of the form
$(\quantq_{1}^{t_{1}} x_{1})\ldots (\quantq_{k}^{t_{k}} x_{k})(\quantq_{1}^{s_{1}} y_{1})\ldots (\quantq_{m}^{s_{m}} y_{m})  
( s_{1}=t_{1} \wedge s_{2}=t_{2}) \; .$
By Lemma \ref{sulten}, we have a formula $\phi'$ of the desired form which is equivalent to $\phi$.
The case when $\phi$ is of the form $\alpha \vee \beta$ is similar. Use Lemma \ref{amanda}(1) in place of Lemma \ref{sulten}.

The theorem follows trivially from the induction hypothesis when $\phi$ is of one of the forms $(\exists v) \alpha$,
$(\forall v \preceq t) \alpha$ and $(\exists v \preceq t) \alpha$.
\qed
\end{proof}

\begin{theorem}[Normal Form Theorem II] \label{xxbispegaard}
Any $\Sigma$-formula $\phi$ is equivalent in $\mathfrak{B}$ to a $\mathcal{L}_{BT}$-formula $\phi'$ of one of  the forms 
$$ \phi' \, \equiv \, (\quantq_{1}^{t_{1}} v_{1} )\ldots (\quantq_{m}^{t_{m}} v_{m} ) \, (s=t) \;\; \mbox{ or }\;\; 
\phi' \, \equiv \, (\exists v)(\quantq_{1}^{t_{1}} v_{1} )\ldots (\quantq_{m}^{t_{m}} v_{m} ) \, (s=t)$$
where $t_{1},..,t_{m}, s,t$ are $\mathcal{L}_{BT}$-terms and 
$\quantq_{j}^{t_{j}} v_{j} \in \lbrace  \exists v_{j} \sqsubseteq t_{j} ,  \forall v_{j} \sqsubseteq t_{j} \rbrace$ for $j= 1, \ldots , m$.  
\end{theorem}

\begin{proof}
Proceed by induction on the structure of $\phi$. This proof is similar to the proof of Theorem \ref{bispegaard}.
 A formula of the form $(\forall x \sqsubseteq t)(\exists y)\alpha$
 is equivalent (in $\mathfrak{B}$) to a formula of the form $(\exists z)(\forall x \sqsubseteq t)(\exists y \sqsubseteq z)\alpha$,
a formula of
the form $(\exists x \sqsubseteq t)(\exists y)\alpha$ is equivalent
  to a formula of  the form $(\exists y)\alpha (\exists x \sqsubseteq t)$,
and a formula of
the form $(\exists x)(\exists y)\alpha$ is equivalent
  to a formula of  the form $(\exists z)(\exists x \sqsubseteq z)(\exists y \sqsubseteq z)\alpha$.
Thus, the resulting normal form will contain maximum one unbounded existential quantifier.
\qed
\end{proof}

\begin{corollary}
The fragment  $\msigf{n}{m}{0}{D}$ is decidable (for any $n,m\in\nat$).
\end{corollary}

\begin{proof}
By  Theorem \ref{bispegaard}, any $\sigf{n}{m}{0}$-sentence is equivalent in $\mathfrak{D}$ to a sentence of the normal  form
$\exists v_{1} \ldots  v_{k} [s=t]$ (regard the bounded existential quantifiers as unbounded). 
The transformation of a $\sigf{n}{m}{0}$-formula into an equivalent formula (in $\mathfrak{D}$) of normal form is constructive.
 Makanin \cite{makanin}   has proved that it is decidable whether an equation on the form 
$$a_{n}x_{n}\ldots a_{1}x_{1}a_{0} = b_{m}y_{m}\ldots b_{1}y_{1}b_{0}$$
where $a_{1},...,a_{n}, b_{1},...,b_{m} \in \{ \mathbf{0}, \mathbf{1}\}^*$, has a solution in $\{ \mathbf{0}, \mathbf{1}\}^*$. 
It follows that the fragment
$\msigf{n}{m}{0}{D}$ is decidable.
\qed
\end{proof}

We have not been able to prove that any $\sigf{n}{m}{0}$-sentence is equivalent in $\mathfrak{B}$ 
to a sentence of the form $\exists v_{1} \ldots \exists v_{k} [s=t]$.
See the comment immediately before Lemma \ref{tyggegummi}.
Thus, we cannot use  Makanin's \cite{makanin} result to prove that  the fragment $\msigf{n}{m}{0}{B}$ is decidable.

\paragraph{Open Problem:} Is the fragment  $\msigf{n}{m}{0}{B}$  decidable (for any $n,m\in\nat$)?

\section{ Undecidable Fragments}

\begin{definition}
{\em Post's Correspondence Problem}, henceforth $\xpcp$, is given by
\begin{itemize}
\item Instance: a list of pairs $\langle b_{1}, b_{1}^{\prime} \rangle,\ldots , \langle b_{n}, b_{n}^{\prime} \rangle$ where 
$b_i,b_i^{\prime}\in \lbrace \boldsymbol{0}, \boldsymbol{1} \rbrace^{*}$
\item Solution:  a finite nonempty sequence $i_{1},...,i_{m}$ of indexes such that 
$$b_{i_{1}} b_{i_{2}}\ldots  b_{i_{m}} = b_{i_{1}}^{ \prime } b_{i_{2}}^{ \prime }\ldots  b_{i_{m}}^{ \prime } \; . $$
\end{itemize}

We define the map  $N: \lbrace \boldsymbol{0}, \boldsymbol{1} \rbrace^{*} \to \lbrace \boldsymbol{0}, \boldsymbol{1} \rbrace^{*}$  by 
$N(\varepsilon) = \varepsilon$,  
$N( \boldsymbol{0}) = \boldsymbol{0} \boldsymbol{1} \boldsymbol{0}$,
$N( \boldsymbol{1}) = \boldsymbol{0} \boldsymbol{1}^{2} \boldsymbol{0}$,   
$N( b \boldsymbol{0}) = N(b) N(\boldsymbol{0})$ and
$N( b \boldsymbol{1}) = N(b) N(\boldsymbol{1} )$.
\end{definition}

It is proved in Post \cite{post} that $\xpcp$  is undecidable. The proof of the next lemma is left to the reader.

\begin{lemma}\label{temedmelk}
The instance $\langle b_{1}, b_{1}^{\prime} \rangle,\ldots , \langle b_{n}, b_{n}^{\prime} \rangle$ of $\xpcp$  has a solution
iff the instance $\langle N(b_{1}), N(b_{1}^{\prime}) \rangle,\ldots , \langle N(b_{n}), N(b_{n}^{\prime}) \rangle$ has a solution.
\end{lemma}

We will now explain the ideas behind our proofs of the next few theorems.
Given the lemma above, it is not very hard to see 
that an instance  $\langle g_{1}, g_{1}^{\prime} \rangle,\ldots , \langle g_{n}, g_{n}^{\prime} \rangle$  of $\xpcp$ 
has a solution iff there exists a bit string of the form
\begin{align*}
\boldsymbol{0} \boldsymbol{1}^{5} \boldsymbol{0} N(a_{1}) \boldsymbol{0} \boldsymbol{1}^{4} \boldsymbol {0} N(b_{1}) \boldsymbol{0} \boldsymbol{1}^{5} \boldsymbol {0} \ldots 
N(a_{m}) \boldsymbol{0} \boldsymbol{1}^{4} \boldsymbol {0} N(b_{m}) \boldsymbol{0} \boldsymbol{1}^{5} \boldsymbol {0} \tag{*}
\end{align*}
where
\begin{enumerate}
\item[(A)] $N(a_{m}) = N(b_{m})$ 
\item[(B)] $N(a_{1}) = g_{j}$ and $N(b_{1}) = g_{j}^{\prime}$ for some $1\leq j \leq n$ 
\item[(C)] $N(a_{k+1}) = N(a_{k}) N(g_{j})$ and $N(b_{k+1}) = N(b_{k}) N(g_{j}^{\prime} )$ for some $1\leq j \leq n$. 
\end{enumerate}
We also see that an instance  $\langle g_{1}, g_{1}^{\prime} \rangle,\ldots , \langle g_{n}, g_{n}^{\prime} \rangle$  of $\xpcp$ 
has a solution iff there exists a bit string  $s$ of the form (*)  that satisfies
\begin{enumerate}
\item[(a)] there is $j\in\{1,\ldots,n\}$ such that 
$\boldsymbol{0} \boldsymbol{1}^{5} \boldsymbol{0} N(g_{j}) 
\boldsymbol{0} \boldsymbol{1}^{4} \boldsymbol{0} N(g_{j}^{\prime} ) \boldsymbol{0} \boldsymbol{1}^{5} \boldsymbol{0}$
is an initial segment of $s$  
\item[(b)]  if  
$$\boldsymbol{0} \boldsymbol{1}^{5} \boldsymbol{0} N(a) 
\boldsymbol{0} \boldsymbol{1}^{4} \boldsymbol{0} N(b) \boldsymbol{0} \boldsymbol{1}^{5} \boldsymbol{0}$$
is a substring of $s$, then  either  $N(a)=N(b)$, or there is   $j\in\{1,\ldots,n\}$ such that 
\[
\boldsymbol{0} \boldsymbol{1}^{5} \boldsymbol{0} N(a)N(g_{j}) 
\boldsymbol{0} \boldsymbol{1}^{4} \boldsymbol{0} N(b )N(g_{j}^{\prime}) \boldsymbol{0} \boldsymbol{1}^{5} \boldsymbol{0} 
\]
is a substring of $s$.  
\end{enumerate}

In the proof of Theorem \ref{teutendmelk} we give a formula which is true in $\mathfrak{D}$ iff there exists a string of the form (*)  that satisfies  (A), (B) and (C).
In the proof of Theorem \ref{ateutendmelk}  we  give  formulas  which are true in
 $\mathfrak{B}$ iff there exists a string of the form (*) that satisfies (a) and (b).
In order to improve the readability of our formulas, we will write $\skilla$ in place of the biteral $\overline{ \boldsymbol{0} \boldsymbol{1}^{5} \boldsymbol{0} }$
and $\skillb$ in place of the biteral $\overline{ \boldsymbol{0} \boldsymbol{1}^{4} \boldsymbol{0} }$.

\begin{theorem}\label{teutendmelk}
The fragment $\msigf{3}{0}{2}{D}$  is undecidable.
\end{theorem}

\begin{proof}
Let $\psi(x) \equiv (\forall z \preceq x) (  z    \overline{  \boldsymbol{1}^{4} } \not\preceq x)$.
Observe that $\psi$ contains one bounded universal quantifier. Observe  that $\psi(\overline{b})$ is true in $\mathfrak{D}$ iff
the bit string $b$ does not contain 4 consecutive ones. Furthermore, let 
 $\phi_{n} (x_{1},...,x_{n},y_{1},...,y_{n})\equiv$ 
\begin{align*} 
(\exists u) & \Big( \ 
 \left( \bigvee_{j=1}^{n}  \skilla   x_{j}  \skillb  y_{j}  \skilla  \preceq u \ \right)   \;\; \wedge \;\;  
\\
 & (\forall v \preceq u) \Big[   \ v \skilla \not\preceq u 
 \;\;  \vee \;\; v \skilla = u   \;\;  \vee \;\; 
 (\exists w_1, w_2) \Big\{ \ 
v    \skilla    w_1    \skillb    w_2    \skilla \preceq u  
\; \wedge  \;  \\ 
&     \psi(w_1 w_2) \; \wedge \; \left[\; w_1 = w_2 \; \vee \; 
\left( \bigvee_{j=1}^{n} 
v    \skilla    w_1     
\skillb    w_2     
\skilla    
 w_1    x_{j}    
\skillb    w_2    y_{j}    
\skilla \preceq u \right)   \right] \Big\}  \ \Big] \  \Big) \; .
\end{align*}

Let $\langle g_{1}, g_{1}^{\prime} \rangle,\ldots , \langle g_{n}, g_{n}^{\prime} \rangle$ be an instance of $\xpcp$.
We have
$$
\mathfrak{D} \models \phi_{n}( \overline{N(g_{1})},\ldots ,\overline{N(g_{n})}, 
\overline{N(g_{1}^{\prime})},\ldots ,\overline{N(g_{n}^{\prime} )} )  
$$
iff there exists a bit sting of the form (*) that satisfies (A), (B) and (C) iff the instance $\langle g_{1}, g_{1}^{\prime} \rangle,\ldots , \langle g_{n}, g_{n}^{\prime} \rangle$ has a solution. 
Furthermore, $\phi_n$ is a $\sigf{3}{0}{2}$-formula. It follows that the fragment $\msigf{3}{0}{2}{D}$  is undecidable.
\qed
\end{proof}

\begin{theorem}\label{ateutendmelk}
The fragments $\msigf{1}{2}{1}{B}$ and $\msigf{1}{0}{2}{B}$ are undecidable.
\end{theorem}

\begin{proof}
Let $\vec{x} = x_{1},\ldots ,x_{n}$, let  $\vec{y}= y_{1},\ldots ,y_{n}$ and let
$$
\alpha(\vec{x},\vec{y},z) \;\; \equiv  \;\;\left( \bigvee_{j=1}^{n}  \; \skilla    x_{j}    \skillb    y_{j}    \skilla \sqsubseteq z \; \wedge\; 
  0   \skilla    x_{j}    \skillb    y_{j}    \skilla \not\sqsubseteq z \;\wedge \;
  1   \skilla    x_{j}    \skillb    y_{j}    \skilla \not\sqsubseteq z   \right)  \; .
$$
Consider the  $\sigf{1}{2}{1}$-formula $\psi_{n}(\vec{x},\vec{y}) \equiv$  
\begin{align*}
 (\exists u) & \Big( \ \alpha(\vec{x},\vec{y},u) \; \wedge \\ &
(\forall v \sqsubseteq u) \Big[ \ 
    \skilla    v    \skilla \not\sqsubseteq u  
\;\; \vee \;\; \overline{ \boldsymbol{1}^{5}  } \sqsubseteq v 
  \;\; \vee \;\;
 (\exists w_{1}, w_{2} \sqsubseteq v)  \Big\{ \ v =  w_{1}    \skillb    w_{2} \\ &  
    \wedge \, 
   \overline{ \boldsymbol{1}^{4} } \not\sqsubseteq w_{1} \, \wedge \,
   \overline{ \boldsymbol{1}^{4}  } \not\sqsubseteq w_{2} \,  \wedge \, 
\left[ \ w_1 = w_2 \, \vee \, 
  \left( \bigvee_{j=1}^{n}   \skilla    w_1    x_{j}    \skillb    w_2    y_{j}    \skilla \sqsubseteq u \right) \ \right]     \   \Big\} 
   \ \Big] \  \Big)  
\end{align*}
and  consider the  $\Sigma_{1}^{1,0 ,2 }$-formula $\gamma_{n} (\vec{x},\vec{y})\equiv $
\begin{align*}
(\exists u)  & \Big( \ \alpha(\vec{x},\vec{y},u) \; \wedge \;
(\forall w_{1}, w_{2} \sqsubseteq u) 
\Big\{ \   
 \skilla    w_{1}    \skillb    w_{2}    \skilla \not\sqsubseteq u \;\;  \vee \;\;
     \overline{  \boldsymbol{1}^{4} } \sqsubseteq w_{1}    w_{2}     
 \\ & \;\;\; \;\;\; \;\;\; \;\;\; \;\;\;\;\;\;\;\;\;\;\;\; \;\;\; \;\;\; \;\;\; \;\;\; \;\;\; \vee \;\; 
   \ w_{1} = w_{2} \; \vee \; 
   \left( \bigvee_{j=1}^{n}  
\skilla    w_{1}    x_{j}    
\skillb    w_{2}    y_{j}    
\skilla  \sqsubseteq u  \right)  
 \ \Big\}  \  \Big)  .
\end{align*}

Let $\langle g_{1}, g_{1}^{\prime} \rangle,\ldots , \langle g_{n}, g_{n}^{\prime} \rangle$ be an instance of $\xpcp$.
We have
$$
\mathfrak{B} \models \psi_{n}( \overline{N(g_{1})},\ldots ,\overline{N(g_{n})}, 
\overline{N(g_{1}^{\prime})},\ldots ,\overline{N(g_{n}^{\prime} )} )  
$$
iff 
$$
\mathfrak{B} \models \gamma_{n}( \overline{N(g_{1})},\ldots ,\overline{N(g_{n})}, 
\overline{N(g_{1}^{\prime})},\ldots ,\overline{N(g_{n}^{\prime} )} )  
$$
iff
there exists a bit sting of the form (*) that satisfies (a) and (b) iff the instance 
$\langle g_{1}, g_{1}^{\prime} \rangle,\ldots, \langle g_{n}, g_{n}^{\prime} \rangle$ has a solution. 
It follows that the  fragments $\msigf{1}{2}{1}{B}$ and $\msigf{1}{0}{2}{B}$ are undecidable.
\qed
\end{proof}

The proof of the next theorem is based on the following idea:
 The instance  $\langle g_{1}, g_{1}^{\prime} \rangle,\ldots, \langle g_{n}, g_{n}^{\prime} \rangle$  of $\xpcp$ 
has a solution iff there exists a bit string of the form
\begin{multline*}
\boldsymbol{0} \boldsymbol{1}^{5} \boldsymbol{0}
 N(a_{1}) \boldsymbol{0} \boldsymbol{1}^{4} \boldsymbol {0} N(b_{1}) \boldsymbol{0} \boldsymbol{1}^{6}\boldsymbol {0}
 N(a_{2}) \boldsymbol{0} \boldsymbol{1}^{4} \boldsymbol {0} N(b_{2}) \boldsymbol{0} \boldsymbol{1}^{7} \boldsymbol {0} \ldots \\
\ldots \boldsymbol{0} \boldsymbol{1}^{5+m-1} \boldsymbol {0}
N(a_{m}) \boldsymbol{0} \boldsymbol{1}^{4} \boldsymbol {0} N(b_{m}) \boldsymbol{0} \boldsymbol{1}^{5+m} \boldsymbol {0} 
\end{multline*}
with the properties (A), (B) and (C) given above.

\begin{theorem}\label{gronnte}
The fragment $\msigf{4}{1}{1}{D}$  is undecidable.
\end{theorem}

\begin{proof}
Let $\bang{k}\equiv \overline{\boldsymbol{0} \boldsymbol{1}^{k} \boldsymbol {0}}$. The $\sigf{4}{1}{1}$-formula
\begin{align*} 
(\exists u)  \Big( \ &
 \left( \bigvee_{j=1}^{n}  \bang{5}   x_{j}  \bang{4}  y_{j}  \bang{6}  \preceq u \ \right)   \;\; \wedge \;\;  
 (\forall v \preceq u) \Big[   \ v \overline{\boldsymbol{1}^{5} \boldsymbol {0}} \not\preceq u 
 \;\;  \vee \;\; v  = 0   \;\;  \vee \;\; \\
& (\exists w_1, w_2, y)  (\exists z\preceq v)\Big\{ \
v= z0y\overline{\boldsymbol{1}^{5} \boldsymbol {0}}w_1\bang{4}w_201y \; \wedge \; 1y=y1
\; \wedge  \;  \\ 
& \;\;\;\;\;\;\;\;\;\;\;\;\;\;\;\;\;\;\;\;\; \left[\; w_1 = w_2 \; \vee \; 
\left( \bigvee_{j=1}^{n} 
v    \overline{\boldsymbol{1}^{5} \boldsymbol {0}}   w_1    x_{j}    
\bang{4}    w_2    y_{j} 011 y \overline{\boldsymbol{1}^{5} \boldsymbol {0}}
\preceq u 
\right)   \right] \Big\}  \ \Big] \  \Big) 
\end{align*}
yields the desired statement. Note that $y$ is a solution of the equation $\boldsymbol{1}y=y \boldsymbol{1}$ 
iff $y\in\{ \boldsymbol{1}\}^*$.
\qed
\end{proof}

\section{Proof of Theorem \ref{bsigcomp}: $\Sigma$-Completeness of $B$}

\label{bsigcompproof}

\begin{lemma} \label{gammelbtre}
$$  B_1, B_2, B_4 
\vdash \forall x[ \ x 0 \neq e \wedge   x1 \neq e \ ] \; . $$
\end{lemma}

\begin{proof}
We reason in an arbitrary model for $\{ B_1, B_2, B_4 \}$.
Let $x$ be an arbitrary element in the universe.
Assume $x0=e$. Then $1(x0)=1e$. By $B_1$, we have $1(x0)=1$.
By $B_2$, we have $(1x)0=1$. By $B_1$, we have  $(1x)0=e1$.
This contradicts $B_4$. This proves that $x0\neq e$.
A symmetric argument shows that  $x1\neq e$. 
This proves that
$$  B_1, B_2, B_4 
\models  \forall x[ \ x 0 \neq e \wedge   x1 \neq e \ ]  \; . $$
The lemma follows by the Completeness Theorem for first-order logic. 
\qed
\end{proof}

\begin{lemma} \label{lemmaaen}
For any variable-free $\mathcal{L}_{BT}$-term $t$ there exists a biteral $b$ such that 
$B\vdash t=b$. Furthermore,  we have
\[\mathfrak{B} \models t_{1} = t_{2} \ \Rightarrow \  B\vdash t_{1} = t_{2} \]
for any variable-free $\mathcal{L}_{BT}$-terms $t_{1}$ and $t_{2}$.
\end{lemma}

\begin{proof}
We proceed by induction on the structure of $t$ to show that there exists a biteral $b$ such that $B\vdash t=b$.

If $t \equiv e$, let $b\equiv e$. Then $B \vdash e=e$. 

If $t \equiv 0$, let $b\equiv e \circ 0$. By $B_1$, we have $B \vdash 0= e\circ 0$.

If $t \equiv 1$, let $b\equiv e \circ 1$. By $B_1$, we have $B \vdash 1 = e\circ 1$.

Suppose $t\equiv t_{1} \circ t_{2}$.
Furthermore,  suppose there exist  biterals $b_{1}$ and $b_{2}$ such that $B\vdash t_{1}=b_{1}$ and $B\vdash t_{2}=b_{2}$. 
Then $B\vdash t_{1} \circ t_{2} =b_{1} \circ b_{2}$. 
We note that $b_{1} \circ b_{2}$ is of the form 
\[
    ( \ (\ldots (e\circ c_{1}) \circ \ldots ) \circ c_{n} \ ) \;\;\; \circ \;\;\; ( \ (\ldots (e\circ d_{1}) \circ \ldots ) \circ d_{m} \ )  
\]
where each $c_{i}$ and each $d_{j}$ is $0$ or $1$. 
Let $\mathfrak{A} \models B$. Then $\mathfrak{A} \models t_{1} \circ t_{2} = b_{1} \circ b_{2}$. By $B_2$ we have
\[
\mathfrak{A} \models b_{1} \circ b_{2} = 
( (\ldots ( (\ldots (e\circ c_{1}) \circ \ldots ) \circ c_{n}) \circ (e\circ d_{1}) ) \circ \ldots ) \circ d_{m})  .
\]
By $B_1$ we have
\[
\mathfrak{A} \models (b_{1} \circ b_{2}) = 
( (\ldots (  (\ldots (e\circ c_{1}) \circ \ldots ) \circ c_{n}) \circ d_{1} ) \circ \ldots ) \circ d_{m}) .
\]
Let 
\[b \equiv ( (\ldots (  (\ldots (e\circ c_{1}) \circ \ldots ) \circ c_{n}) \circ d_{1} ) \circ \ldots ) \circ d_{m}) .
\]
Then $b$ is a biteral and $\mathfrak{A} \models b_{1} \circ b_{2} = b$. 
Since $\mathfrak{A} \models t_{1}\circ t_{2}=b_{1}\circ b_{2}$, we have $\mathfrak{A} \models t_{1}\circ t_{2}=b$. 
Since $\mathfrak{A}$ is an arbitrary model for $B$,  we have $B\models t_{1}\circ t_{2}=b$, and then, 
by the Completeness Theorem for first-order logic, we have $B\vdash t_{1}\circ t_{2}=b$. 

This proves that there for any variable-free term $t$ there exists a biteral $b$ such that $B\vdash t=b$.

Let $t_{1}$ and $t_{2}$ be $\mathcal{L}_{BT}$-terms  such that $\mathfrak{B} \models t_{1} = t_{2}$. 
Then  there exist biterals $b_{1}$ and $b_{2}$ such that 
\[B\vdash t_{1}=b_{1} \mbox{    and     } B\vdash t_{2}=b_{2} .\]
Since $\mathfrak{B} \models B $, we have
\[\mathfrak{B} \models t_{1}=b_{1} \mbox{    and     } \mathfrak{B} \models t_{2}=b_{2} .
\]
and thus we also have $\mathfrak{B} \models  b_{1}=b_{2}$.
Since each element in $\{ \boldsymbol{0},  \boldsymbol{1}\}^*$ is mapped to a unique biteral, 
it follows that $b_{1}$ is the same biteral as $b_{2}$. Thus,  $B\vdash b_{1}=b_{2}$. 
Thus, $B\vdash t_{1}=t_{2}$. 
\qed
\end{proof}

\begin{lemma} \label{lemmaato}
We have
\[\mathfrak{B} \models \neg b_{1} = b_{2}  \ \Rightarrow \  B\vdash \neg b_{1} = b_{2}   \] 
for any biterals $b_{1}$ and $b_{2}$.
Furthermore, we have
\[\mathfrak{B} \models  \neg t_{1} = t_{2}  \ \Rightarrow \  B\vdash  \neg t_{1} = t_{2}  \]
for any variable-free $\mathcal{L}_{BT}$-terms $t_{1}$ and $t_{2}$.
\end{lemma}

\begin{proof}
Let $b_{1}$ and $b_{2}$ be biterals such that $\mathfrak{B} \models \neg b_{1} = b_{2}$. 
We proceed by induction on the structure of $b_{2}$ to show that $B\vdash \neg b_{1} = b_{2}$.

If $b_{2} \equiv e  $, then $b_{1}\equiv b \circ 0$ or $b_{1}\equiv b \circ 1$ for some biteral $b$. In either case, by 
Lemma \ref{gammelbtre}, we have $B\vdash (\neg b \circ 0 = e) \wedge (\neg b \circ 1 = e) $.

Suppose $b_{2} \equiv t \circ 0$. Furthermore,  suppose by induction hypothesis that 
\begin{align*} 
\mathfrak{B} \models \neg b=t \; \Rightarrow \;
B \vdash \neg b=t \tag{\mbox{IH}}
\end{align*} 
for any biteral $b$. We proceed by induction on $b_{1}$.  
If $b_{1} \equiv e$, we have  $B\vdash \neg e= t \circ 0$ by Lemma \ref{gammelbtre}. 
If $b_{1} \equiv b \circ 0$, then $\mathfrak{B} \models \neg b= t$. 
By (IH), we have $B\vdash \neg b= t$. By $B_3$, we have $B\vdash \neg b \circ 0 = t \circ 0$.
If $b_{1} \equiv b \circ 1$, we have   $B\vdash \neg b \circ 1 = t \circ 0$ by $B_4$.

This case when $b_{2} \equiv t \circ 1$ is symmetric to the  case when $b_{2} \equiv t \circ 0$.

This proves that
\begin{align*}
\mathfrak{B} \models \neg b_{1}=b_{2} \ \Rightarrow \ B\vdash \neg b_{1}=b_{2} \; . \tag{*}
\end{align*}

Now, suppose $t_{1}$ and $t_{2}$ are variable-free $\mathcal{L}_{BT}$-terms such that $\mathfrak{B} \models \neg t_{1} = t_{2}$.  
By Lemma \ref{lemmaaen}, there exist biterals $b_{1}$ and $b_{2}$ such that 
$B\vdash t_{1}=b_{1} \wedge t_{2} = b_{2}$. As  $\mathfrak{B} \models B$, we have 
$\mathfrak{B} \models t_{1}=b_{1} \wedge t_{2} = b_{2}$. It follows that  $\mathfrak{B} \models \neg b_{1} = b_{2}$. 
By (*), we have  $B\vdash \neg b_{1} = b_{2}$, and thus we also have  $B\vdash \neg t_{1} = t_{2}$.
\qed
\end{proof}

\begin{lemma} \label{lemmaatre}
We have
\[\mathfrak{B} \models b_{1} \sqsubseteq b_{2}  \ \Rightarrow \  B\vdash b_{1} \sqsubseteq b_{2}  \]
for any biterals $b_{1}$ and $b_{2}$. 
Furthermore, we have
\[\mathfrak{B} \models  t_{1} \sqsubseteq t_{2}  \ \Rightarrow \  B\vdash t_{1} \sqsubseteq t_{2}  \]
for any variable-free $\mathcal{L}_{BT}$-terms $t_{1}$ and $t_{2}$. 
\end{lemma}

\begin{proof}
We prove this lemma by induction on the structure of $b_{2}$.

If $b_{2}\equiv e$ and $\mathfrak{B} \models b_{1} \sqsubseteq b_{2} $, then $b_{1}$ is $e$.
By $B_5$, we have $B \vdash e\sqsubseteq e$. 

If $b_{2}\equiv e\circ 0$ and $\mathfrak{B} \models b_{1} \sqsubseteq b_{2} $, then $b_{1}$ is $e$ or $e\circ 0$. In either case, by Lemma
\ref{lemmaaen} and  $B_6$, we have $B \vdash b_{1} \sqsubseteq b_{2}$. 

If $b_{2}\equiv e\circ 1$ and $\mathfrak{B} \models b_{1} \sqsubseteq b_{2} $, then $b_{1}$ is $e$ or $e\circ 1$. In either case, by Lemma 
\ref{lemmaaen} and $B_7$, we have $B \vdash b_{1} \sqsubseteq b_{2}$. 

Suppose $b_{2}\equiv e\circ 0\circ t\circ 0$.
Furthermore, suppose by induction hypothesis that we for any biteral $s$ have
\begin{itemize}
\item $\mathfrak{B} \models s\sqsubseteq e\circ 0\circ t \ \Rightarrow \ B \vdash s\sqsubseteq e\circ 0\circ t$ 
\item $\mathfrak{B} \models s\sqsubseteq t\circ 0 \ \Rightarrow \ B \vdash s\sqsubseteq t\circ 0$. 
\end{itemize}
Let $\mathfrak{B} \models b_{1} \sqsubseteq b_{2}$. Then we have 
$$\mathfrak{B} \models b_{1}= e\circ 0\circ t\circ 0 \;  \vee \; b_{1} \sqsubseteq e\circ 0\circ t \; \vee \;
 b_{1} \sqsubseteq t\circ 0 \; .$$ 
By our induction hypothesis and Lemma \ref{lemmaaen}, 
we have 
$$B \vdash b_{1}= e\circ 0\circ t\circ 0 \; \vee \; b_{1} \sqsubseteq e\circ 0\circ t \; \vee \;  b_{1} \sqsubseteq t\circ 0 \; .$$
By $B_1$ and $B_8$, we have $B \vdash b_{1} \sqsubseteq e\circ 0\circ t\circ 0$. 

The case when  $b_{2}\equiv e\circ 0\circ t\circ 1$, the case when $b_{2}\equiv e\circ 1\circ t\circ 0$ and 
the case when $b_{2}\equiv e\circ 1\circ t\circ 1$ are 
handled similarly using $B_9$, $B_{10}$ and $B_{11}$, respectively, in place of $B_8$.  
This proves that we have
\begin{align*}
\mathfrak{B} \models b_{1} \sqsubseteq b_{2} \ \Rightarrow \ B \vdash b_{1} \sqsubseteq b_{2} \; . \tag{*}
\end{align*}
for any biterals $b_{1}, b_{2}$

Suppose $t_{1}$ and $t_{2}$ are variable-free $\mathcal{L}_{BT}$-terms such that $\mathfrak{B} \models t_{1} \sqsubseteq t_{2}$.  By Lemma 
\ref{lemmaaen}, there exists biteral $b_{1}$ and $b_{2}$ such that 
$B\vdash t_{1}=b_{1} \wedge t_{2} = b_{2}$. Since $\mathfrak{B} \models B$, we also have 
$\mathfrak{B} \models t_{1}=b_{1} \wedge t_{2} = b_{2}$. Hence, $\mathfrak{B} \models b_{1} \sqsubseteq b_{2}$. 
By (*), we have $B\vdash b_{1} \sqsubseteq b_{2}$, and thus we also have $B\vdash t_{1} \sqsubseteq t_{2}$.
\qed
\end{proof}

\begin{lemma} \label{lemmaafire}
Let $\phi(x)$ be an $\mathcal{L}_{BT}$ formula such that 
\[
\mathfrak{B} \models \phi(b) \ \Rightarrow \ B \vdash \phi(b)  \]
for any biteral $b$.
Then
\[
\mathfrak{B} \models (\forall x  \sqsubseteq b)\phi(x)  \ \Rightarrow \ 
B \vdash (\forall x  \sqsubseteq b)\phi(x)  \]
 for any biteral $b$.
\end{lemma}

\begin{proof}
We proceed by induction on $b$. 

Let $b\equiv e$. By $B_5$, we have
$$
\mathfrak{B} \models (\forall x  \sqsubseteq e)\phi(x) \ \Leftrightarrow  \ 
\mathfrak{B} \models \phi(e) \ 
\Rightarrow   \ B \vdash \phi(e)  \ \Rightarrow \ B \vdash (\forall x  \sqsubseteq e)\phi(x) \;  .
$$

Let $b \equiv e\circ 0$. By $B_1$ and $B_6$, we have
\begin{multline*}
\mathfrak{B} \models (\forall x  \sqsubseteq e\circ 0)\phi(x) \  \Leftrightarrow \ 
\mathfrak{B} \models \phi(e) \wedge \phi(0) \ \Rightarrow \ 
B \vdash \phi(e) \wedge \phi(0)  \\
\  \Rightarrow \ B \vdash (\forall x  \sqsubseteq e\circ 0)\phi(x) .
\end{multline*}

Let $b \equiv e\circ 1$. This case is symmetric to the case $b \equiv e\circ 0$.
Use  $B_7$ in place of $B_6$.

Let $b\equiv e\circ 0\circ t\circ 0$. Suppose by induction hypothesis (IH) that 
\begin{itemize}
\item  $\mathfrak{B} \models (\forall x  \sqsubseteq e\circ 0\circ t)\phi(x) 
\ \Rightarrow \ B \vdash (\forall x  \sqsubseteq e\circ 0\circ t)\phi(x)$
\item  $\mathfrak{B} \models (\forall x  \sqsubseteq t\circ 0)\phi(x) 
\ \Rightarrow \  B \vdash (\forall x  \sqsubseteq t\circ 0 )\phi(x)$.
\end{itemize}
Then, by the assumption on $\phi$ given in our lemma, we have
\renewcommand{\arraystretch}{1.4}
$$\begin{array}{cl}
 \mathfrak{B} \models (\forall x  \sqsubseteq e\circ 0\circ t \circ 0 )\phi(x) &  \\
  \Updownarrow & \\
 \mathfrak{B} \models (\forall x  \sqsubseteq e\circ 0\circ t)\phi(x) \; \wedge \; 
(\forall x  \sqsubseteq  t\circ 0)\phi(x) \; \wedge \; \phi(e\circ 0\circ t\circ 0) \;\;\;\;\; & \\
 \Downarrow & \mbox{\small (IH)}  \\
  B \vdash (\forall x  \sqsubseteq e\circ 0\circ t)\phi(x) \; \wedge \;
(\forall x  \sqsubseteq  t\circ 0)\phi(x) \; \wedge \; \phi(e\circ 0\circ t\circ 0)  & \\
 \Downarrow & \mbox{\small ($B_8$)}  \\
  B \vdash (\forall x  \sqsubseteq e\circ 0\circ t \circ 0) \phi(x) \;  . & 
\end{array}$$
\renewcommand{\arraystretch}{1.0}

The case when  $b\equiv e\circ 0\circ t\circ 1$, the case when $b\equiv e\circ 1\circ t\circ 0$ and 
the case when $b\equiv e\circ 1\circ t\circ 1$ are handled similarly using $B_9$, $B_{10}$ and $B_{11}$, respectively, in place of $B_8$. 
\qed
\end{proof}

\begin{lemma} \label{lemmaafem}
We have 
\[\mathfrak{B} \models \neg b_{1} \sqsubseteq b_{2}  \ \Rightarrow \  B\vdash \neg b_{1} \sqsubseteq b_{2}  \]
for any biterals $b_{1}, b_{2}$. 
Furthermore, we have
\[\mathfrak{B} \models  \neg t_{1} \sqsubseteq t_{2}  \ \Rightarrow \  B\vdash \neg t_{1} \sqsubseteq t_{2}  \]
for any variable-free $\mathcal{L}_{BT}$-terms $t_{1}, t_{2}$.
\end{lemma}

\begin{proof}
We proceed by induction on $b_{2}$. 

If $b_{2}\equiv e$ and $\mathfrak{B} \models \neg b_{1} \sqsubseteq e$, 
then $\mathfrak{B} \models \neg b_{1} = e$. By Lemma \ref{lemmaato}, we have $B \vdash \neg b_{1} = e$. 
By $B_5$, we have $B \vdash \neg b_{1} \sqsubseteq e$.

If $b_{2}\equiv e\circ 0$ and $\mathfrak{B} \models \neg b_{1} \sqsubseteq e\circ 0$, 
then $\mathfrak{B} \models \neg b_{1} = e \ \wedge \ \neg b_{1} = 0$. By Lemma \ref{lemmaato}, we have 
$B \vdash \neg b_{1} = e \ \wedge \ \neg b_{1} = 0$. 
By $B_6$, we have $B \vdash \neg b_{1} \sqsubseteq e\circ 0$.  

If $b_{2}\equiv e\circ 1$ and $\mathfrak{B} \models \neg b_{1} \sqsubseteq e\circ 1$, 
then $\mathfrak{B} \models \neg b_{1} = e \ \wedge \ \neg b_{1} = 1$. By Lemma \ref{lemmaato}, we have
$B \vdash \neg b_{1} = e \ \wedge \ \neg b_{1} = 1$. 
By $B_7$, we have $B \vdash \neg b_{1} \sqsubseteq e\circ 1$.

Let $b_{2}\equiv e\circ 0\circ t\circ 0$.
Suppose by induction hypothesis that we have
\begin{itemize}
\item $\mathfrak{B} \models \neg s\sqsubseteq e\circ 0\circ t \ \Rightarrow \ B \vdash \neg s\sqsubseteq e\circ 0\circ t$
\item $\mathfrak{B} \models \neg s\sqsubseteq t\circ 0 \; \Rightarrow \ B \vdash \neg s\sqsubseteq t\circ 0$ 
\end{itemize}
 for any biteral $s$.
Let $\mathfrak{B} \models \neg b_{1} \sqsubseteq b_{2}$. 
Then 
$$\mathfrak{B} \models \neg b_{1} \sqsubseteq e\circ 0\circ t \ \wedge \ \neg b_{1} \sqsubseteq t\circ 0 \ \wedge \ 
\neg b_{1} = e\circ 0\circ t \circ 0 \; . $$ 
By our induction hypothesis and Lemma \ref{lemmaato}, 
we have 
$$B\vdash \neg b_{1} \sqsubseteq e\circ 0\circ t \ \wedge \ \neg b_{1} \sqsubseteq t\circ 0 \ \wedge \ 
\neg b_{1} = e\circ 0\circ t \circ 0 \; . $$ 
By  $B_8$, we have $B \vdash \neg b_{1} \sqsubseteq e\circ 0\circ t\circ 0$.

The case when  $b \equiv e\circ 0\circ t\circ 1$, the case when $b \equiv e\circ 1\circ t\circ 0$ and 
the case when $b \equiv e\circ 1\circ t\circ 1$ 
are handled similarly using $B_9$, $B_{10}$ and $B_{11}$, respectively, in place of $B_8$.
Thus, we conclude that we have
\begin{align*}
\mathfrak{B} \models \neg b_{1} \sqsubseteq b_{2} \ \Rightarrow \ B\vdash \neg b_{1} \sqsubseteq b_{2} \; . \tag{*}
\end{align*}
for any biterals $b_{1}, b_{2}$.

Let $t_{1}$ and $t_{2}$ be variable-free $\mathcal{L}_{BT}$-terms such that 
$\mathfrak{B} \models \neg t_{1} \sqsubseteq t_{2}$.  By Lemma \ref{lemmaaen}, we have biterals $b_{1}$ and $b_{2}$ such that 
$B\vdash t_{1}=b_{1} \wedge t_{2} = b_{2}$. 
Since $\mathfrak{B} \models B$, we also have
 $\mathfrak{B} \models t_{1}=b_{1} \wedge t_{2} = b_{2}$. Hence $\mathfrak{B} \models \neg b_{1} \sqsubseteq b_{2}$. 
By (*), we have $B\vdash \neg b_{1} \sqsubseteq b_{2}$, and thus  $B\vdash \neg t_{1} \sqsubseteq t_{2}$.
\qed
\end{proof}

We are now prepared to prove Theorem \ref{bsigcomp}
We proceed by induction on the structure of the $\Sigma$-sentence $\phi$. 

If $\phi$ is an atomic formula or the negation of an atomic formula, then applications of Lemma \ref{lemmaaen}, Lemma \ref{lemmaato}, 
Lemma \ref{lemmaatre} or Lemma \ref{lemmaafem} give 
\[
\mathfrak{B} \models \phi \ \Rightarrow \ B \vdash \phi.
\]

Let $\phi \equiv \alpha \vee \beta $. 
Assume $\mathfrak{B} \models \alpha \vee \beta$. Then we have $\mathfrak{B} \models \alpha$  or  $\mathfrak{B} \models  \beta$.
We can w.l.o.g.\ assume that $\mathfrak{B} \models \alpha$. By our induction hypothesis, we have $B\vdash \alpha$.
Finally, as $\alpha \vee \beta$ follows logically from $\alpha$, we conclude that $B\vdash \alpha \vee \beta$. 

The case when  $\phi\equiv \alpha \wedge \beta$ is similar to the case when $\phi \equiv \alpha \vee \beta$.

Let $\phi \equiv (\exists x ) \alpha(x)$. The induction hypothesis yields
\[
\mathfrak{B} \models \alpha(t) \ \Rightarrow \ B \vdash \alpha(t)
\]
for any variable-free term $t$.
Now assume that $\mathfrak{B} \models (\exists x ) \alpha(x)$.
Then there exists a biteral $b$ such that  
$\mathfrak{B} \models  \alpha( b )$. By our induction hypothesis, we have
$B\vdash  \alpha( b )$. As $\vdash (\exists x ) \alpha(x)$ follows logically from 
$\alpha( b )$, we have $B\vdash (\exists x ) \alpha(x)$.

Let $\phi \equiv (\exists x  \sqsubseteq t)\alpha(x)$ where $t$ is a variable-free term.
The induction hypothesis yields 
\[
\mathfrak{B} \models \alpha(t) \ \Rightarrow \ B \vdash \alpha(t)
\]
for any variable-free term $t$. Assume $\mathfrak{B} \models (\exists x  \sqsubseteq t)\alpha(x)$
Then there exists  biteral $b$ such that  $\mathfrak{B} \models  b  \sqsubseteq t$ and
$\mathfrak{B} \models  \alpha( b )$. By Lemma \ref{lemmaatre}, we have $B\vdash  b  \sqsubseteq t$.
By our induction hypothesis, we have $B\vdash  \alpha( b )$. It follows that $B\vdash (\exists x  \sqsubseteq t)\alpha(x)$.

Let $\phi\equiv (\forall x  \sqsubseteq t) \alpha(x)$ where $t$ is a variable-free term.
The induction hypothesis yields  
\[
\mathfrak{B} \models \alpha(t) \ \Rightarrow \ B \vdash \alpha(t)
\]
for any variable-free term $t$. Assume $\mathfrak{B} \models (\forall x  \sqsubseteq t) \alpha(x)$.
By Lemma \ref{lemmaaen}, there exists a biteral $b$ such that $B \vdash t=b$.
Obviously,  $\mathfrak{B} \models (\forall x  \sqsubseteq b) \alpha(x)$.
By Lemma \ref{lemmaafire} and our induction hypothesis, we have
$B\vdash (\forall x  \sqsubseteq b ) \alpha(x)$. Finally, as $B \vdash t=b$, we have 
$B\vdash (\forall x  \sqsubseteq t) \alpha(x)$.

This completes the proof of Theorem \ref{bsigcomp}.

\section{Proof of Theorem \ref{dsigcomp}: $\Sigma$-Completeness of $D$}

\label{dsigcompproof}

We now proceed to prove that $D$ is  $\Sigma$-complete.
Recall that the first four axioms of $D$ are the same as the first four axioms of $B$.

\begin{lemma} \label{lemmaben}
For any variable-free $\mathcal{L}_{BT}$-term $t$ there exists a biteral $b$ such that 
$D\vdash t=b$. Furthermore, we have
\[\mathfrak{D} \models t_{1} = t_{2} \ \Rightarrow \  D\vdash t_{1} = t_{2}  \]
for any variable-free $\mathcal{L}_{BT}$-terms $t_{1}$ and $t_{2}$.
\end{lemma}

\begin{proof}
This proof is identical to the proof of  Lemma \ref{lemmaaen}.
\qed
\end{proof}

\begin{lemma} \label{lemmabto}
For any biterals $b_{1}$ and $b_{2}$
\[\mathfrak{D} \models \neg b_{1} = b_{2}  \ \Rightarrow \  D\vdash \neg b_{1} = b_{2}  \; . \] 
Furthermore, for any variable-free $\mathcal{L}_{BT}$-terms $t_{1}$ and $t_{2}$
\[\mathfrak{D} \models  \neg t_{1} = t_{2}  \ \Rightarrow \  D\vdash  \neg t_{1} = t_{2}  \; . \]
\end{lemma}

\begin{proof}
This proof is identical to the proof of  Lemma  \ref{lemmaato}.
\qed
\end{proof}

\begin{lemma} \label{lemmabtre}
We have
\[\mathfrak{D} \models b_{1} \preceq b_{2}  \ \Rightarrow \  D\vdash b_{1} \preceq b_{2}  \]
for any biterals $b_{1}$ and $b_{2}$. 
Furthermore, we have
\[\mathfrak{D} \models  t_{1} \preceq t_{2}  \ \Rightarrow \  D\vdash t_{1} \preceq t_{2}  \] 
 for any variable-free $\mathcal{L}_{BT}$-terms $t_{1}$ and $t_{2}$.
\end{lemma}

\begin{proof}
We proceed by induction on $b_{2}$.

If $b_{2}\equiv e$ and $\mathfrak{D} \models b_{1} \preceq b_{2} $, then $b_{1}$ is $e$.
By $D_5$, we have $D \vdash e\preceq e$. 

Let $b_{2}\equiv t\circ 0$. 
Assume $\mathfrak{D} \models b_{1} \preceq b_{2}$. Then 
$\mathfrak{D} \models  b_{1} \preceq t \vee b_{1} = b_{2}$. 
By the induction hypothesis and Lemma \ref{lemmaben}, we have $D \vdash b_{1} \preceq t \vee b_{1} = b_{2}$.
By $D_6$, we have 
$D \vdash b_{1} \preceq b_{2}$.

The case when Let $b_{2}\equiv t\circ 1$ is similar to the case  $b_{2}\equiv t\circ 1$.  
Use $D_7$ in place of $D_6$.

Thus, we conclude that
\[
\mathfrak{D} \models b_{1} \preceq b_{2} \ \Rightarrow \ D \vdash b_{1} \preceq b_{2} \; .
\]
holds  for any biterals $b_{1}, b_{2}$. It is easy to see that also the second part 
of the theorem holds (see the proof Lemma \ref{lemmaatre}).
\qed
\end{proof}

\begin{lemma} \label{lemmabfire}
Let $\phi( x )$ be an $\mathcal{L}_{BT}$-formula such that we have
\[
\mathfrak{D} \models \phi(b) \ \Rightarrow \ D \vdash \phi(b)  \]
for any biteral $b$.
Then, we also have
\[
\mathfrak{D} \models (\forall  x  \preceq b)\phi( x )  \ \Rightarrow \ 
D \vdash (\forall  x  \preceq b)\phi( x )  \]
for any biteral $b$.
\end{lemma}

\begin{proof}
We prove the lemma by induction on $b$. 

Let $b\equiv e$. We have
\begin{multline*}
\mathfrak{D} \models (\forall  x  \preceq e)\phi( x ) \ \Leftrightarrow   \ 
\mathfrak{D} \models \phi(e) 
 \ \Rightarrow  \ D \vdash \phi(e)
\ \Rightarrow  \ D \vdash (\forall  x  \preceq e)\phi( x )\; .
\end{multline*}
The last implication holds by $D_5$.

Let $b\equiv t\circ 0$. Assume by induction hypothesis that
\begin{align*}
\mathfrak{D} \models (\forall  x  \preceq t)\phi( x ) 
\ \Rightarrow \ 
D \vdash (\forall  x  \preceq t)\phi( x )  .
\end{align*}
By the assumption on $\phi$ and the induction hypothesis, we have 
\begin{align*}
\mathfrak{D} \models (\forall  x  \preceq t \circ 0 )\phi( x ) \ \Leftrightarrow & \ 
\mathfrak{D} \models (\forall  x  \preceq t)[\phi( x )] \wedge \phi(t\circ 0)  
\\
\Rightarrow & \  D \vdash (\forall  x  \preceq t )[\phi( x )] \wedge \phi(t \circ 0) \\
\Rightarrow & \  D \vdash (\forall  x  \preceq t \circ 0) \phi( x ) \; .
\end{align*}
The last implication holds by $D_6$. 

The case $b\equiv t\circ 1$ is similar to the case $b\equiv t\circ 0$.
Use $D_7$ in place of $D_6$.
\qed
\end{proof}

\begin{lemma} \label{lemmabfem}
We have
\[\mathfrak{D} \models \neg b_{1} \preceq b_{2}  \ \Rightarrow \  
D\vdash \neg b_{1} \preceq b_{2}  \] 
for any biterals $b_{1}$ and $b_{2}$.
Furthermore, we have
\[\mathfrak{D} \models  \neg t_{1} \preceq t_{2}  \ \Rightarrow \ 
 D\vdash \neg t_{1} \preceq t_{2} \]
for any variable-free $\mathcal{L}_{BT}$-terms $t_{1}$ and $t_{2}$.
\end{lemma}

\begin{proof}
We proceed by induction on $b_{2}$. 

Let $b_{2}\equiv e$. Assume $\mathfrak{D} \models \neg b_{1} \sqsubseteq e$. 
Then $\mathfrak{D} \models \neg b_{1} = e$. By Lemma \ref{lemmabto}, we have $D \vdash \neg b_{1} = e$. 
By $D_5$, we have $D \vdash \neg b_{1} \preceq e$. 

Let $b_{2}\equiv t\circ 0$.
Assume $\mathfrak{D} \models \neg b_{1} \preceq b_{2}$. Then 
$\mathfrak{D} \models \neg b_{1} \preceq t \wedge \neg b_{1} = t$. 
By the induction hypothesis and Lemma \ref{lemmabto}, we have
$D \vdash \neg b_{1} \preceq t \wedge \neg b_{1} = t$. 
By $D_6$, we have 
$D \vdash \neg b_{1} \preceq b_{2}$.

The case  $b_{2}\equiv t\circ 1$ is similar to the case 
$b_{2}\equiv t\circ 0$. Use $D_7$ in place of $D_6$.

This proves that
\[
\mathfrak{D} \models \neg b_{1} \preceq b_{2} \ \Rightarrow \ D\vdash \neg b_{1} \preceq b_{2} \; .
\]
holds for any biterals $b_{1}, b_{2}$. It is easy to see that also the second part 
of the theorem holds (see e.g.\ the proof Lemma \ref{lemmaafem}).
\qed
\end{proof}

 Theorem \ref{dsigcomp} is  proved by induction on the structure of the $\Sigma$-sentence $\phi$.
Proceed as in the proof of Theorem  \ref{bsigcomp} (see Section \ref{bsigcompproof}) and use the lemmas above.


\begin{thebibliography}{AaKrRuud}

\bibitem{bs}
  B\"uchi, J. R.  and  Senger, S.:  
{\em Coding in the existential theory of concatenation.} Arch. math. Logik \textbf{26} (1986/7), 101-106.


\bibitem{day}
 Day, J.,  Ganesh,  V.,  He, P.,   Manea, F. and  Nowotka, D.:  {\em The satisfiability of extended word equations: 
The boundary between decidability and undecidability.} arXiv:1802.00523 (2018). 

\bibitem{ganesh}
 Ganesh,  V.,  Minnes,  M.,  Solar-Lezama,  A. and  Rinard, M. C.:  {\em Word equations with length constraints: What’s decidable?} 
In: Biere A., Nahir A., Vos T. (eds) Hardware and Software: Verification and Testing. 
HVC 2012. Lecture Notes in Computer Science, vol 7857,  pp. 209–226. Springer, Berlin, Heidelberg.




\bibitem{grz}
Grzegorczyk, A.: {\em Undecidability without arithmetization.}
Studia Logica
\textbf{79} (2005),  163-230.


\bibitem{zd}
Grzegorczyk, A. and Zdanowski, K.:
{\em Undecidability and concatenation.} pp. 72-91 in
``Andrzej Mostowski and Foundational Studies''
 (eds. by Ehrenfeucht et al.), IOS, Amsterdam, 2008. 

\bibitem{halfon}
Halfon, S.,  Schnoebelen, P. and Zetzsche G: {\em Decidability, complexity, and expressiveness of first-order 
logic over the subword ordering.} 
In {\em Proc. LICS 2017.} IEEE Computer Society, 1–12.

\bibitem{hori}
Horihata, Y.: {\em Weak theories of concatenation and arithmetic.}
Notre Dame Journal of Formal Logic, \textbf{53} (2012), 203-222.




\bibitem{karh}
Karhum\"aki, J.,  Mignosi, F. and  Plandowski, W.:  {\em The expressibility of languages and relations by word equations.} Journal of the ACM  \textbf{47}  (2000), 483–505.








\bibitem{leary}
Leary, C.  and   Kristiansen, L.: {\em A friendly introduction to mathematical logic.}
 2nd  Edition,  Milne  Library, SUNY Geneseo, Geneseo, NY, 2015.



\bibitem{makanin}
Makanin, G. S.: {\em  The problem of solvability of equations in a free semigroup.} 
Mathematics of the USSR-Sbornik  \textbf{32} (1977), 129-198. 

\bibitem{post}
Post, E. L.:
{\em A variant of a recursively unsolvable problem.}
 Bulletin of the American Mathematical Society \textbf{52} (1946), 264-268.

\bibitem{quine}
 Quine, W. V.:  {\em Concatenation as a basis for arithmetic.}
The Journal of Symbolic Logic
\textbf{11} (1946),  105-114. 



\bibitem{senger}
Senger, S.: {\em The existential theory of concatenation over a finite alphabet.} PhD dissertation, Purdue University (1982).


\bibitem{tarski}
Tarski, A.: {\em Der Wahrheitsbegriff in den formalisierten Sprachen.}
Studia Philosophica \textbf{1} (1935),  261-405. 

\bibitem{utarski}
Tarski, A.: {\em Undecidable theories.}
Studies in Logic and the Foundations of Mathematics. North-Holland Publishing Company, Amsterdam, 1953. In collaboration with
A.~Mostowski and R.~M.~Robinson.






\bibitem{visser}
Visser, A.: {\em Growing commas. A study of sequentiality and concatenation.}
Notre Dame Journal of Formal Logic, \textbf{50} (2009), 61-85.


\end{thebibliography}
\end{document}